\documentclass[11pt,a4paper]{amsart}
\usepackage{amssymb,amsmath}
\usepackage{comment}
\usepackage{graphicx}
\usepackage{subcaption}

\def\({\left(}
\def\){\right)}
\newcommand{\niton}{\not\owns}

\def\Cal{\mathcal}

\def\eb{\varepsilon}

\def\R {\mathbb{R}}

\newcommand{\be}{\begin{equation} }
\newcommand{\ee}{\end{equation} }

\def\Q {{\mathbb Q}}

\def \and{\qquad\text{and}\qquad}

\def\Bbb{\mathbb}

\def\({\left(}
\def\){\right)}

\def\eb{\varepsilon}
\def\Cal{\mathcal}
\def\eb{\varepsilon}

\def\R {\mathbb{R}}

\def\Q {{\mathbb Q}}

\def\<{\left<}
\def\>{\right>}

\def \and{\qquad\text{and}\qquad}

\def\Bbb{\mathbb}

\newtheorem{proposition}{Proposition}[section]
\newtheorem{theorem}[proposition]{Theorem}
\newtheorem{corollary}[proposition]{Corollary}
\newtheorem{lemma}[proposition]{Lemma}
\theoremstyle{definition}
\newtheorem{definition}[proposition]{Definition}
\newtheorem{remark}[proposition]{Remark}

\numberwithin{equation}{section}

\def\be{\begin{equation}}
\def\ee{\end{equation}}
\def\bp{\begin{proof}}
\def\ep{\end{proof}}

\def \no#1#2#3 {{\bf #1} (#3), #2.}
\def \eds#1#2#3 {#1, #2, #3.}

\title[Convergence of lattice sums]
{Cesaro summation by spheres of lattice sums and Madelung constants }
\author[B. Galbally  and  S. Zelik]
{ Benjamin Galbally${}^1$  and Sergey Zelik${}^{1,2}$}

\address{${}^1$
Department of Mathematics,\newline \indent
University of Surrey, GU27XH,
Guildford,  UK}

\address{${}^2$ \phantom{e}School of Mathematics and Statistics, Lanzhou University, Lanzhou
\newline\indent 730000,
P.R. China}
\email{bg00298@surrey.ac.uk}
\email{s.zelik@surrey.ac.uk}
\begin{document}
\begin{abstract}
We study  convergence of 3D lattice sums via  expanding spheres. It is well-known that, in contrast to summation via  expanding cubes, the expanding spheres method may lead to formally divergent series (this will be so e.g. for the classical NaCl-Madelung constant). In the present paper we prove that these series remain convergent in Cesaro sense. For the case of second order Cesaro summation, we present an elementary proof of  convergence and the proof for first order Cesaro summation is more involved and is based on the Riemann localization for multi-dimensional Fourier series.
\end{abstract}
\thanks{The first author has been partially supported by the LMS URB grant 1920-04. The second author is partially supported by the EPSRC grant EP/P024920/1}
\maketitle
\tableofcontents
\section{Introduction}\label{s0}
Lattice sums of the form
\begin{equation}\label{0.lattice}
\sum_{(n,k,m)\in\Bbb Z^3}\frac{e^{i(nx_1+kx_2+mx_3)}}{(a^2+n^2+k^2+m^2)^s}
\end{equation}
and various their extensions naturally appear in many branches of modern analysis including analytic number theory (e.g. for study the number of lattice points in spheres or balls), analysis of PDEs (e.g. for constructing Green functions for various differential operators in periodic domains, finding best constants in interpolation inequalities, etc.), harmonic analysis as well as in applications, e.g. for computing the electrostatic potential of a single ion in a crystal (the so-called Madelung constants), see \cite{Flap,MFS,BDZ,Bor13,mar,mar2000,Ram21,ZI} and references therein.  For instance, the classical Madelung constant for the NaCl crystal is given by
\begin{equation}\label{0.M}
M=\sideset{}{'}\sum_{(i,j,k)\in \Bbb Z^3}\frac{(-1)^{i+j+k}}{(i^2+j^2+k^2)^{1/2}},
\end{equation}
where the index ${'}$ means that the sum does not contain the term which corresponds to $(i,j,k)=0$.
\par
The common feature of series \eqref{0.lattice} and \eqref{0.M} is that the decay rate of the terms is not strong enough to provide absolute convergence, so they are often only conditionally convergent and their convergence/divergence strongly depends on the method of summation. The typical methods of summation are  summation by expanding cubes/rectangles or summation by expanding spheres, see sections \S\ref{s1} and \S\ref{s2} for definitions and \cite{Bor13} for more details. For instance, when summation by expanding spheres is used, the formula for the Madelung constant has an especially elegant form
\begin{equation}\label{2.Ms}
M=\sum_{n=1}^\infty (-1)^n\frac{r_3(n)}{\sqrt{n}},
\end{equation}
where $r_3(n)$ is the number of integer point in a sphere of radius $\sqrt{n}$. Exactly this formula is commonly used in physical literature although it has been known for more than 70 years that series \eqref{2.Ms} is {\it divergent}, see \cite{Emer}.  Thus, one should either switch from expanding spheres to expanding cubes/rectangles for summation of \eqref{0.M} (which is suggested to do e.g. in \cite{Bor13} and where such a convergence problem does not appear)  or to use more advanced methods for summation of \eqref{2.Ms}, for instance Abel or Cesaro summation. Surprisingly, the possibility to justify \eqref{2.Ms} in such a way is not properly studied (although there are detailed results concerning Cesaro summation for different methods, e.g. for the so called summation by diamonds, see \cite{Bor13}) and the main aim of the present  notes is to cover this gap.
\par
Namely, we will study the following generalized Madelung constants:
\begin{equation}\label{2.Mg}
M_{a,s}=\sideset{}{'}\sum_{(i,j,k)\in \Bbb Z^3}\frac{(-1)^{i+j+k}}{(a^2+i^2+j^2+k^2)^s}=\sum_{n=1}^\infty(-1)^n\frac{r_3(n)}{(a^2+n)^s},
\end{equation}
where $a\in\R$ and $s>0$ and the sum in the RHS is understood in the sense of Cesaro (Cesaro-Riesz) summation of order $\kappa$, see Definition \ref{Def2.Cesaro} below. Our presentation of the main result consists of two parts.
\par
First, we present a very elementary proof of convergence for second order Cesaro summation which is based only on counting the number of lattice points in spherical layers by volume comparison arguments. This gives the following result
\begin{theorem}\label{Th0.c2} Let $a\in\R$ and $s>0$. Then
\begin{equation}
M_{a,s}=\lim_{N\to\infty}\sum_{n=1}^N(-1)^n\(1-\frac nN\)^2 \frac{r_3(n)}{(a^2+n)^s}.
\end{equation}
In particular, the limit in the RHS exists.
\end{theorem}
Second, we establish the convergence for the first order Cesaro summation.
\begin{theorem}\label{Th0.c1} Let $a\in\R$ and $s>0$. Then
\begin{equation}
M_{a,s}=\lim_{N\to\infty}\sum_{n=1}^N(-1)^n\(1-\frac nN\) \frac{r_3(n)}{(a^2+n)^s}.
\end{equation}
In particular, the limit in the RHS exists.
\end{theorem}
In contrast to Theorem \ref{Th0.c2}, the proof of this result is more involved and is based on an interesting connection between the convergence of lattice sums and  Riemann localization for multiple Fourier series, see section \S\ref{s22} for more details.  Note that Theorem \ref{Th0.c2} is a formal corollary of Theorem \ref{Th0.c1}, but we prefer to keep both of them not only since the proof of Theorem \ref{Th0.c2} is essentially simple, but also since it possesses extensions to other methods of summation, see the discussion in section \S\ref{s3}. Also note that the above convergence results have mainly theoretical interest since much more effective formulas
for Madelung constants are available for practical computations, see \cite{Bor13} and references therein.
\par
The paper is organized as follows. Some preliminary results concerning lattice sums and summation by rectangles are collected in \S\ref{s1}. The proofs of Theorems \ref{Th0.c2} and \ref{Th0.c1} are given in sections \S\ref{s21} and \S\ref{s22} respectively. Some discussion around the obtained results, their possible generalizations and numerical simulations are presented in section \S\ref{s3}.

\section{Preliminaries}\label{s1}
In this section, we recall  standard results about lattice sums and prepare some technical tools which will be used in the sequel. We start with the simple lemma which is however crucial for what follows.
\begin{lemma}\label{Lem1.block} Let the function $f:\R^3\to\R$ be 3 times continuously differentiable in a cube $Q_{I,J,K}:=[I,I+1]\times[J,J+1]\times[K,K+1]$. Then
\begin{multline}\label{1.E}
\min_{x\in Q_{2I,2J,2K}}\{-\partial_{x_1}\partial_{x_2}\partial_{x_3}f(x)\}\le\\\le E_{I,J,K}(f):=\sum_{i=2I}^{2I+1}\sum_{j=2J}^{2J+1}
\sum_{k=2K}^{2K+1}(-1)^{i+j+k}f(i,j,k)\le\\\le \max_{x\in Q_{2I,2J,2K}}\{-\partial_{x_1}\partial_{x_2}\partial_{x_3}f(x)\}.
\end{multline}
\end{lemma}
\begin{proof} Indeed, it is not difficult to check using the Newton-Leibnitz formula that
$$
E_{I,J,K}(f)=-\int_0^1\int_0^1\int_0^1 \partial_{x_1}\partial_{x_2}\partial_{x_3}f(2I+s_1,2J+s_2,2K+s_3)\,ds_1\,ds_2\,ds_3
$$
and this formula gives the desired result.
\end{proof}
\noindent A typical example of the function $f$ is the following one
\begin{equation}\label{1.pol}
f_{a,s}(x)=(a^2+|x|^2)^s,\ \ |x|^2=x_1^2+x_2^2+x_3^2.
\end{equation}
In this case,
$$
\partial_{x_1}\partial_{x_2}\partial_{x_3}f= 8s(s-1)(s-2)x_1x_2x_3(a^2+|x|^2)^{s-3/2}
$$
and, therefore,
\begin{equation}\label{1.bet}
|E_{I,J,K}(f)|\le C(a^2+I^2+J^2+K^2)^{s-\frac32}.
\end{equation}
 One more important property of the function \eqref{1.pol} is that the term $E_{I,J,K}$ is sign-definite in the octant $I,J,K\ge0$.
 \par
 At the next step, we state a straightforward extension of the integral comparison principle to the case of multi-dimensional series. We recall that, in one dimensional case, for a positive monotone decreasing  function $f:[A,B]\to\R$,  $A,B\in\Bbb Z$, $B>A$, we have
 $$
 f(B)+\int_A^{B}f(x)\,dx\le \sum_{n=A}^Bf(n)\le f(A)+\int_{A}^{B}f(x)\,dx
$$
which, in turn, is an immediate corollary of the estimate
$$
f(n+1)\le\int_n^{n+1}f(x)\,dx\le f(n).
$$
\begin{lemma}\label{Lem1.int} Let the continuous function $f:\R^3\setminus\{0\}\to\R_+$ be such that
\begin{equation}\label{1.good}
C_2\max_{x\in Q_{i,j,k}} f(x)\le \min_{x\in Q_{i,j,k}}f(x)\le C_1\max_{x\in Q_{i.j,k}} f(x),
\end{equation}
$(i,j,k)\in\Bbb Z^3$ and  the constants $C_1$ and $C_2$ are positive and are independent of $Q_{i,j,k}\niton 0$.
Let also $\Omega\subset\R^3$ be a domain which does not contain $0$ and
\begin{equation}\label{1.lat}
\Omega_{lat}:=\{(i,j,k)\in\Bbb Z^3:\,\exists Q_{I,J,K}\subset\Omega,\ \ (i,j,k)\in Q_{I,J,K},\ 0\notin Q_{I,J,K}\}.
\end{equation}
Then,
\begin{equation}\label{1.comp}
\sum_{(i,j,k)\in\Omega_{lat}}f(i,j,k)\le C\int_\Omega f(x)\,dx,
\end{equation}
where the constant $C$ is independent of $\Omega$ and $f$. If assumption \eqref{1.good}
is satisfied for all $(I,J,K)$, the condition $0\notin\Omega$ and $0\notin Q_{I,J,K}$ can be removed.
\end{lemma}
\begin{proof} Indeed,  assumption \eqref{1.good} guarantees that
\begin{equation}\label{1.mult}
C_2\int_{Q_{I,J,K}}f(x)\,dx\le f(i,j,k)\le C_1\int_{Q_{I,J,K}}f(x)\,dx
\end{equation}
for all $Q_{I,J,K}$ which do not contain zero and all $(i,j,k)\in Q_{I,J,K}\cap\Bbb Z^3$. Since any point $(i,j,k)\in\Bbb Z$ can belong no more than $8$ different cubes $Q_{I,J,K}$, \eqref{1.mult} implies \eqref{1.comp} (with the constant $C=8C_1$) and finishes the proof of the lemma.
\end{proof}
We will mainly use this lemma for functions $f_{a,s}(x)$ defined by \eqref{1.pol}. It is not difficult to see that these functions satisfy assumption \eqref{1.good}. For instance, this follows from the obvious estimate
$$
|\nabla f_{a,s}(x)|\le \frac{C_{s}}{\sqrt{a^2+|x|^2}}f_{a,s}(x)
$$
and the mean value theorem. Moreover, if $a\ne0$, condition \eqref{1.good} holds for $Q_{i,j,k}\owns 0$ as well. As a corollary, we get the following estimate for summation "by spheres":
\begin{multline}\label{1.as}
\sideset{}{'}\sum_{(i,j,k)\in B_n\cap\Bbb Z^3} f_{a,s}(i,j,k)\le C_s\int_{x\in B_n\setminus B_1}(a^2+x^2)^{s}\,dx\le\\\le 4\pi C_s\int_1^{\sqrt n} R^2(a^2+R^2)^s\,dR\le 4\pi C_s\int_1^{\sqrt n} R(a^2+R^2)^{s-1/2}\,dR=\\=\frac {4\pi C_s}{2s+3}\((a^2+n)^{s+3/2}-(a^2+1)^{s+3/2}\),
\end{multline}
where $B_n:=\{x\in\R^3\,:\,|x|^2\le n\}$ and $\sum'$ means that $(i,j,k)=0$ is ex\-clu\-ded. Of course, in the case $s=-\frac32$, the RHS of \eqref{1.as} reads as $2\pi C_s\ln\frac{a^2+n^2}{a^2+1}$. In particular, if $s>\frac32$, passing to the limit $n\to\infty$ in \eqref{1.as}, we see that
\begin{equation}\label{1.simple}
\sideset{}{'}\sum_{(i,j,k)\in\Bbb Z^3}\frac1{(a^2+i^2+j^2+k^2)^s}=
\sideset{}{'}\sum_{(i,j,k)\in\Bbb Z^3}f_{a,-s}(i,j,k)\le \frac {C_s}{(a^2+1)^{s-\frac32}}.
\end{equation}
 Thus, the series in the LHS is absolutely convergent if $s>\frac32$ and its sum tends to zero as $a\to\infty$. It is also well-known that condition $s>\frac32$ is sharp and the series is
 divergent if $s\le \frac32$.
\par
We also mention that Lemmas \ref{Lem1.block} and \ref{Lem1.int} are stated for 3-dimensional case just for simplicity. Obviously, their analogues hold for any dimension. We will use this observation later.
\par
We now turn to the alternating version of lattice sums \eqref{1.simple}
\begin{equation}\label{1.main}
M_{a,s}:=\sideset{}{'}\sum_{(i,j,k)\in\Bbb Z^3}\frac{(-1)^{i+j+k}}{(a^2+i^2+j^2+k^2)^s}
\end{equation}
which is the main object of study in these notes. We recall that, due to \eqref{1.simple}, this series is absolutely convergent for $s>\frac32$, so the sum is independent of the method of summation. In contrast to this, in the case $0<s\le\frac32$, the convergence is not absolute and depends strongly to the method of summation, see \cite{Bor13} and references therein for more details. Note also that $M_{a,s}$ is analytic in $s$ and, similarly to the classical Riemann zeta function, can be extended to a holomorphic function on $\Bbb C$ with a pole at $s=0$, but this is beyond the scope of our paper, see e.g. \cite{Bor13} for more details. Thus, we are assuming from now on that $0<s\le\frac32$. We start with the most studied case of summation by expanding rectangles/parallelograms.
\begin{definition} Let $\Pi_{I,J,K}:=[-I,I]\times[-J,J]\times[-K,K]$, $I,J,K\in\Bbb N$, and
$$
S_{\Pi_{I,J,K}}(a,s):=\sideset{}{'}\sum_{(i,j,k)\in\Pi_{I,J,K}\cap\Bbb Z^3}\frac{(-1)^{i+j+k}}{(a^2+i^2+j^2+k^2)^s}.
$$
We say that \eqref{1.main} is summable by expanding rectangles if the following triple limit exists and finite
$$
M_{a,s}=\lim_{(I,J,K)\to\infty} S_{\Pi_{I,J,K}}(a,s).
$$
\end{definition}
To study the sum \eqref{1.main}, we combine the terms belonging to cubes $Q_{2i,2j,2k}$ and introduce the partial sums
\begin{equation}
E_{\Pi_{I,J,K}}(a,s):=\sideset{}{'}\sum_{(2i,2j,2k)\in\Pi_{I,J,K}\cap2\Bbb Z^3}E_{i,j,k}(a.s),
\end{equation}
where $E_{i,j,k}(a,s):=E_{i,j,k}(f_{a,-s})$ is defined in \eqref{1.E}.
\begin{theorem} Let $0<s\le \frac32$. Then,
\begin{equation}\label{1.equiv}
\bigg|S_{\Pi_{I,J,K}}(a,s)-E_{\Pi_{I,J,K}}(a,s)\bigg|\le \frac {C_s}{\(a^2+\min\{I^2,J^2,K^2\}\)^{s}},
\end{equation}
where the constant $C_s$ is independent of $a$ and $I,J,K$.
\end{theorem}
\begin{proof} We first mention that, according to Lemma \ref{Lem1.block} and estimate \eqref{1.simple}, we see that
\begin{equation}\label{1.e-conv}
|E_{\Pi_{I,J,K}}(a,s)|\le \frac{C_s}{(a^2+1)^s}
\end{equation}
uniformly with respect to $(I,J,K)$.
\par
The difference between $S_{\Pi_{I,J,K}}$ and $E_{\Pi_{I,J,K}}$ consists of the alternating sum of $f_{a,-s}(i,j,k)$ where $(i,j,k)$ belong to the boundary of $\Pi_{I,J,K}$. Let us write an explicit formula for the case when all $I,J,K$ are even (other cases are considered analogously):
\begin{multline}\label{1.huge}
S_{\Pi_{2I,2J,2K}}(a,s)-E_{\Pi_{2I,2J,2K}}(a,s)=\!\!\!\sideset{}{'}\sum_{\substack{-2J\le j\le-2J\\-2K\le k\le2K}}(-1)^{j+k}f_{a,-s}(2I,j,k)+\\+\sideset{}{'}\sum_{\substack{-2I\le i\le-2I\\-2K\le k\le2K}}(-1)^{i+k}f_{a,-s}(i,2J,k)+\sideset{}{'}\sum_{\substack{-2I\le i\le-2I\\-2J\le j\le2J}}(-1)^{i+j}f_{a,-s}(i,j,2K)-\\-
\sideset{}{'}\sum_{-2I\le i\le-2I}(-1)^{i}f_{a,-s}(i,2J,2K)-\sideset{}{'}\sum_{-2J\le j\le-2J}(-1)^{j}f_{a,-s}(2I,j,2K)-\\-
\sideset{}{'}\sum_{-2K\le k\le-2K}(-1)^{k}f_{a,-s}(2I,2J,k)+f_{a.-s}(2I,2J,2K).
\end{multline}
In the RHS of this formula we see the analogues of lattice sum \eqref{1.main} in lower dimensions one or two and, thus, it allows to reduce the dimension. Indeed, assume that the analogues of estimate \eqref{1.equiv} are already established in one and two dimensions. Then, using the lower dimensional analogue of \eqref{1.e-conv} together with the fact that
$$
f_{a,-s}(2I,j,k)=f_{\sqrt{a^2+4I^2},-s}(i,j),
$$
where we have 2D  analogue of the function $f_{a,-s}$ in the RHS, we arrive at
\begin{multline}\label{1.huge1}
\bigg|S_{\Pi_{2I,2J,2K}}(a,s)-E_{\Pi_{2I,2J,2K}}(a,s)\bigg|\le\\\le
\frac{C_s}{(a^2+4I^2+1)^s}+\frac{C_s}{(a^2+\min\{J^2,K^2\})^s}+
\frac{C_s}{(a^2+4J^2+1)^s}+\\+\frac{C_s}{(a^2+\min\{I^2,K^2\})^s}+
\frac{C_s}{(a^2+4K^2+1)^s}+\frac{C_s}{(a^2+\min\{I^2,J^2\})^s}+\\+
+\frac{C_s}{(a^2+4I^2+4K^2\})^s}+\frac{C_s}{(a^2+4I^2+4J^2\})^s}+
\frac{C_s}{(a^2+4J^2+4K^2\})^s}+\\+
\frac{C_s}{(a^2+4I^2+4J^2+4K^2\})^s}\le \frac{C_s'}{(a^2+\min\{I^2,J^2,K^2\})^s}.
\end{multline}
Since in 1D case the desired estimate is obvious, we complete the proof of the theorem by induction.
\end{proof}
\begin{corollary}\label{Cor1.main} Let $s>0$. Then series \eqref{1.main} is convergent by expanding rectangles and
\begin{equation}\label{1.rep}
M_{a,s}=\sideset{}{'}\sum_{(i,j,k)\in\Bbb Z^3}E_{i,j,k}(a,s).
\end{equation}
In particular, the series in RHS of \eqref{1.rep} is absolutely convergent, so the method of summation for it is not important.
\end{corollary}
Indeed, this fact is an immediate corollary of estimates \eqref{1.equiv}, \eqref{1.bet} and \eqref{1.simple}.

\section{Summation by expanding spheres}\label{s2}
We now turn to summation by expanding spheres. In other words, we want to write the formula \eqref{1.main} in the form
\begin{equation}\label{2.sphere}
M_{a,s}=\lim_{N\to\infty}\sideset{}{'}\sum_{i^2+j^2+k^2\le N}\frac{(-1)^{i+j+k}}{(a^2+i^2+j^2+k^2)^s}.
\end{equation}
Moreover, since $(i+j+k)^2=i^2+j^2+k^2+2(ij+jk+ik)$, we have $(-1)^{i+j+k}=(-1)^{i^2+j^2+k^2}$, so formula \eqref{2.sphere} can be rewritten in the following elegant form
\begin{equation}\label{2.sp}
M_{a,s}=\sum_{n=1}^\infty (-1)^n\frac{r_3(n)}{(a^2+n)^s},
\end{equation}
where $r_3(n)$ is the number of integer points on a sphere of radius $\sqrt{n}$ centered at zero, see e.g. \cite{Ram21} and reference therein for more details about this function. However, the convergence of series \eqref{2.sp} is more delicate. In particular, it is well-known that this series is divergent for $s\le\frac12$, see \cite{Emer,Bor13}. For the convenience of the reader, we give the proof of this fact below.

\begin{lemma}\label{Lem2.div} Let $c>0$ be small enough. Then, there are infinitely many values of $n\in\Bbb N$ such that
\begin{equation}\label{2.bad}
r_3(n)\ge c\sqrt{n}
\end{equation}
and, particularly, series \eqref{2.sp} is divergent for all $s\le\frac12$.
\end{lemma}
\begin{proof} Indeed, by comparison of volumes, we see that the number $M_N$ of integer points in a spherical layer $N\le i^2+j^2+k^2\le 2N$ can be estimated from above as
$$
M_N=\sum_{n=N}^{2N}r_3(n)\ge \frac43\pi\((\sqrt{2N}-\sqrt3)^{3}-(\sqrt{N}+\sqrt3)^{3}\)\ge  cN^{3/2}
$$
for sufficiently small $c>0$. Thus, for every sufficiently big $N\in\Bbb N$, there exists $n\in[N,2N]$ such that $r_3(n)\ge c\sqrt{n}$ and estimate \eqref{2.bad} is verified. The divergence of \eqref{2.sp} for $s\le \frac12$ is an immediate corollary of this estimate since the $n$th term $(-1)^n\frac{r_3(n)}{(a^2+n)^s}$ does not tend to zero under this condition and the lemma is proved.
\end{proof}
\begin{remark} The condition that $c>0$ is small can be removed using more sophisticated methods. Moreover, it is known that the inequality
$$
r_3(n)\ge c\sqrt{n}\ln\ln n
$$
holds for infinitely many values of $n\in\Bbb N$ (for properly chosen $c>0$). On the other hand,
for every $\eb>0$, there exists $C_\eb>0$ such that
$$
r_3(n)\le C_\eb n^{\frac12+\eb},
$$
 see  \cite{Ram21} and references therein. Thus, we cannot establish divergence of \eqref{2.sp} via the $n$th term test if $s>\frac12$. Since this series is alternating, one may expect  convergence for $s>\frac12$. However, the behavior of $r_3(n)$ as $n\to\infty$ is very irregular and, to the best of our knowledge, this convergence is still an open problem for $\frac12<s\le\frac{25}{34}$, see \cite{Bor13} for the convergence in the case $s>\frac{25}{34}$ and related results.
\end{remark}
Thus, one should use weaker concepts of convergence in order to justify equality \eqref{2.sp}. The main aim of these notes is to establish the convergence in the sense of Cesaro.
\begin{definition}\label{Def2.Cesaro} Let $\kappa>0$. We say the  series \eqref{2.sp} is $\kappa$-Cesaro (Cesaro-Riesz) summable if the sequence
$$
C^\kappa_N(a,s):=\sum_{n=1}^N\(1-\frac nN\)^\kappa(-1)^n\frac{r_3(n)}{(a^2+n)^s}
$$
is convergent. Then we write
$$
(C,\kappa)-\sum_{N=1}^\infty (-1)^n\frac{r_3(n)}{(a^2+n)^s}:=\lim_{N\to\infty}C_N^\kappa(a,s).
$$
Obviously, $\kappa=0$ corresponds to the usual summation and if a series is $\kappa$-Cesaro summable, then it is also $\kappa_1$-Cesaro summable for any $\kappa_1>\kappa$, see e.g.~\cite{Ha}.
\end{definition}

\subsection{Second order Cesaro summation}\label{s21} The aim of this subsection is to present a very elementary proof of the fact that the series \eqref{2.sp} is second order Cesaro summable. Namely, the following theorem holds.
\begin{theorem}\label{Th2.2c} Let $s>0$. Then the series \eqref{2.sp} is second order Cesaro summable and
\begin{equation}\label{2.2good}
M_{a,s}=(C,2)-\sum_{N=1}^\infty (-1)^n\frac{r_3(n)}{(a^2+n)^s},
\end{equation}
where $M_{a,s}$ is the same as in \eqref{1.main} and \eqref{1.rep}.
\end{theorem}
\begin{proof}
  For every $N\in\Bbb N$, let us introduce the sets
\begin{equation*}
D_N:=\bigcup\limits_{\substack{(I,J,K)\in2\Bbb Z^3\\Q_{I,J,K}\subset B_N}}Q_{I,J,K},\ \ \ D_N':=B_N\setminus D_N
\end{equation*}
and split the sum $C^2_N(a,s)$ as follows
\begin{multline}
C^2_N(a,s)=\sideset{}{'}\sum_{(i,j,k)\in B_N\cap\Bbb Z^3}\(1-\frac{i^2+j^2+k^2}{N}\)^2\frac{(-1)^{i+j+k}}{(a^2+i^2+j^2+k^2)^s}=\\=\sideset{}{'}\sum_{(i,j,k)\in D_N\cap\Bbb Z^3}\(1-\frac{i^2+j^2+k^2}{N}\)^2\frac{(-1)^{i+j+k}}{(a^2+i^2+j^2+k^2)^s}+\\+
\!\!\!\sideset{}{'}\sum_{(i,j,k)\in D_N'\cap\Bbb Z^3}\!\!\!\(1\!-\!\frac{i^2+j^2+k^2}{N}\)^2\!\!\!\frac{(-1)^{i+j+k}}{(a^2+i^2+j^2+k^2)^s}\!:=
A_N(a,s)+R_N(a,s).
\end{multline}
Let us start with estimating the sum $R_N(a,s)$. To this end we use the elementary fact that
$$
\sqrt{N}-\sqrt3\le \sqrt{i^2+j^2+k^2}\le \sqrt{N}
$$
for all $(i,j,k)\in D_N'$ ($\sqrt{3}$ is the length of the diagonal of the cube $Q_{I,J,K}$). Therefore,
\begin{equation}\label{2.R}
|R_N(a,s)|\le \(1-\frac{(\sqrt{N}-\sqrt3)^2}{N}\)^2\frac{\#\(D'_M\cap\Bbb Z^3\)}{\(a^2+(\sqrt {N}-\sqrt3)^2\)^s}.
\end{equation}
Using again the fact that all integer points of $D'_N$ belongs to the spherical layer
 $\sqrt{N}-\sqrt{3}\le |x|^2\le\sqrt{N}$ together with the volume comparison arguments, we conclude
  that
$$
\#\(D'_M\cap\Bbb Z^3\)\le \frac43\pi\((\sqrt{N}+\sqrt3)^2-(\sqrt{N}-\sqrt3)^2\)\le c_0 N
$$
for some positive $c_0$. Therefore,
\begin{equation}
|R_N(a,s)|\le \frac{C}{N}\frac{c_0N}{\(a^2+(\sqrt {N}-\sqrt3)^2\)^s}=\frac{C}{\(a^2+(\sqrt {N}-\sqrt3)^2\)^s}\to0
\end{equation}
as $N\to\infty$. Thus, the term $R_N$ is not essential and we only need to estimate the sum $A_N$. To this end, we rewrite it as follows
\begin{multline}\label{2.huge2}
A_N(a,s)=\(1-\frac{a^2}N\)^2\sideset{}{'}\sum_{(i,j,k)\in D_N\cap\Bbb Z^3}\frac{(-1)^{i+j+k}}{(a^2+i^2+j^2+k^2)^s}+\\+\frac2N\(1-\frac {a^2}N\)\sideset{}{'}\sum_{(i,j,k)\in D_N\cap\Bbb Z^3}\frac{(-1)^{i+j+k}}{(a^2+i^2+j^2+k^2)^{s-1}}+\\+
\frac1{N^2}\sideset{}{'}\sum_{(i,j,k)\in D_N\cap\Bbb Z^3}\frac{(-1)^{i+j+k}}{(a^2+i^2+j^2+k^2)^{s-2}}=\\=
\(1-\frac{a^2}N\)^2\sideset{}{'}\sum_{(i,j,k)\in \frac12D_{N}\cap\Bbb Z^3}E_{i,j,k}(a,s)+\\+\frac2N\(1-\frac {a^2}N\)\sideset{}{'}\sum_{(i,j,k)\in \frac12D_{N}\cap\Bbb Z^3}E_{i,j,k}(a,s-1)+\\+
\frac1{N^2}\sideset{}{'}\sum_{(i,j,k)\in \frac12D_{N}\cap\Bbb Z^3}E_{i,j,k}(a,s-2).
\end{multline}
From Corollary \ref{Cor1.main}, we know that the first sum in the RHS of \eqref{2.huge2} converges to $M_{a,s}$ as $N\to\infty$.  Using estimates \eqref{1.bet} and \eqref{1.as}, we also conclude that
\begin{equation}
 \bigg|\sideset{}{'}\sum_{(i,j,k)\in \frac12D_N\cap\Bbb Z^3}E_{i,j,k}(a,s-1)\bigg|\le CN^{1-s}
\end{equation}
and
\begin{equation}
 \bigg|\sideset{}{'}\sum_{(i,j,k)\in \frac12D_N\cap\Bbb Z^3}E_{i,j,k}(a,s-2)\bigg|\le CN^{2-s}.
\end{equation}
Thus, two other terms in the RHS of \eqref{2.huge2} tend to zero as $N\to\infty$ and the theorem is proved.
\end{proof}

\subsection{First order Cesaro summation}\label{s22} We may try to treat this case analogously to the proof of Theorem \ref{Th2.2c}. However, in this case, we will have the multiplier $(1-\frac{(\sqrt{N}-\sqrt3)^2}{N})$ without the extra square and this leads to the extra technical assumption $s>\frac12$. In particular, this method does not allow us to establish the convergence for the case of classical NaCl-Madelung constant ($a=0$, $s=\frac12$). In this subsection, we present an alternative method based on the Riemann localization principle for multiple Fourier series which allows us to remove the technical condition $s>\frac12$. The key idea of our method is to introduce the function
\begin{equation}\label{2.F}
M_{a,s}(x):=\sideset{}{'}\sum_{(n,k,l)\in\Bbb Z^3}\frac{e^{i(nx_1+kx_2+lx_3)}}{(a^2+n^2+k^2+l^2)^s}.
\end{equation}
The series is clearly convergent, say,  in $\Cal D'(\Bbb T^3)$ and defines (up to a constant) a fundamental solution for the fractional Laplacian $(a^2-\Delta)^s$ on a torus $\Bbb T^3$ defined on functions with zero mean. Then, at least formally,
$$
M_{a,s}=M_{a,s}(\pi,\pi,\pi)
$$
and justification of this is related to the convergence problem for multi-dimensional Fourier series.
\par
Let $G_{a,s}(x)$ be the fundamental solution for  $(a^2-\Delta)^s$ in the whole space $\R^3$, i.e.
$$
G_{a,s}(x)=-\frac{1}{2^{\frac12+s}\pi^{\frac32}\Gamma(s)}\frac1{|x|^{3-2s}}\Psi(a|x|),\ \ \Psi(z):=z^{\frac32-s}K_{\frac32-s}(z),
$$
where $K_\nu(z)$ is a modified Bessel function of the second kind and $\Gamma(s)$ is the Euler gamma function, see e.g. \cite{SL,Watson}. In particular, passing to the limit $a\to0$ and using that $\Psi(0)=2^{\frac12-s}\Gamma(\frac32-s)$, we get the fundamental solution for the case $a=0$:
$$
G_{0,s}(x)=-\frac{\Gamma(\frac32-s)}{2^{2s}\pi^{\frac32}\Gamma(s)}\,\frac1{|x|^{3-2s}}.
$$
 Then, as known, the periodization of this function will be the fundamental solution on a torus:
\begin{equation}\label{2.Poisson}
M_{a,s}(x)=C_0+\frac1{(2\pi)^3}\sum_{(n,k,l)\in\Bbb Z^3}G_{a,s}\(x-2\pi(n,k,l)\),
\end{equation}
where the constant $C_0$ is chosen in such a way that $M_{a,s}(x)$ has a zero mean on the torus, see \cite{Flap,Trans} and references therein. Recall that, for $a>0$, the function $G_{a,s}(x)$ decays exponentially as $|x|\to\infty$, so the convergence of \eqref{2.Poisson} is immediate (and identity \eqref{2.Poisson} is nothing more than the Poisson Summation Formula applied to \eqref{2.F}). However, when $a=0$, the convergence of \eqref{2.Poisson} is more delicate since $G_{0,s}(x)\sim |x|^{2s-3}$ and the decay rate is not strong enough to get the absolute convergence. Thus, some regularization should be done and the method of summation also becomes important, see \cite{Bor13,CR,mar2000} and reference therein. Recall also that we need to consider the case $s\le\frac12$ only (since for $s>\frac12$, we have  convergence of the first order Cesaro sums by elementary methods).

\begin{lemma}\label{Lem2.Green} Let $0<s<1$. Then
\begin{multline}\label{2.Poisson0}
M_{0,s}(x)=C_0'+\frac1{(2\pi)^3}
G_{0,s}(x)+\\+\frac1{(2\pi)^3}\sideset{}{'}\sum_{(n,k,l)\in\Bbb Z^3}\bigg(G_{0,s}\(x-2\pi(n,k,l)\)-G_{0,s}(2\pi(n,k,l))\bigg),
\end{multline}
where the convergence is understood in the sense of convergence by expanding rectangles and $C'_0$ is chosen in such a way that the mean value of the expression in the RHS is zero.
\end{lemma}
\begin{proof}[Sketch of the proof] Although this result seems well-known, we sketch  below the proof of convergence of the RHS (the equality with the LHS can be established after that in a standard way, e.g. passing to the limit $a\to0$ in \eqref{2.Poisson}).
\par
To estimate the terms in the RHS, we use the following version of a mean value theorem for second differences:
\begin{multline}
f(p+x)+f(p-x)-2f(p)=[f(p+x)-f(p)]-[f(p)-f(p-x)]\\=x\int_0^1(f'(p+\kappa x)-f'(p-\kappa x))\,d\kappa=\\=
2x^2\int_0^1\int_0^1\kappa_1\kappa f''(p+\kappa(1-2\kappa_1)x)\,d\kappa\,d\kappa_1
\end{multline}
applying this formula to the function $G_{0,s}(x)$, we get
\begin{multline*}
\bigg|\sum_{\eb_i=\pm1,\, i=1,2,3 }\bigg(G_{0,s}(2\pi n+\eb_1x_1,2\pi k+\eb_2 x_2,2\pi l+\eb_3x_3)-G_{0,s}(2\pi(n,k,l))\bigg)\bigg|\\\le
C\sum_{i=1}^3\|\partial^2_{x_i}G_{0,s}\|_{C(2\pi(n,k,l)+\Bbb T^3)}\le \frac{C_1}{(n^2+k^2+l^2)^{\frac32-2(s-1)}}.
\end{multline*}
Thus, we see that, if we combine together in the RHS of \eqref{2.Poisson0} the terms corresponding to 8 nodes $(\pm n,\pm k,\pm l)$ (for every fixed $(n,k,l)$), the obtained series will become absolutely convergent (here we use the assumption $s<1$).
\par
It remains to note that the parallelepipeds $\Pi_{N,M,K}$  enjoy the property: $(n,m,k)\in\Pi_{N,M,K}$ implies that all 8 points $(\pm n,\pm m,\pm k)\in \Pi_{N,M,K}$. This implies the convergence by expanding rectangles and finishes the proof of the lemma.
\end{proof}
\begin{corollary}\label{Cor2.Grsm} Let $0<s<\frac32$ and $a>0$ or $a=0$ and $0<s<1$. Then, the function $M_{a,s}(x)$ is $C^\infty(\Bbb T^3\setminus\{0\})$ and $G_{a,s}(x)\sim \frac C{|x|^{3-2s}}$ near zero. In particular. $M_{a,s}\in L^{1+\eb}(\Bbb T^3)$ for some positive $\eb=\eb(s)$.
\end{corollary}
\begin{proof} Indeed, the infinite differentiability follows from \eqref{2.Poisson} and \eqref{2.Poisson0} since differentiation of $G_{a,s}(x)$ in $x$ can only improve the rate of convergence. In addition, $M_{a,s}(x)-\frac1{(2\pi)^3}G_{a,s}(x)$ is smooth on the whole $\Bbb T^3$, so $M_{a,s}$ belongs to the same Lebesgue space $L^p$ as the function $|x|^{2s-3}$.
\end{proof}
\begin{remark} The technical assumption $s<1$ can be removed using the fact that $(-\Delta)^{s_1}(-\Delta)^{s_2}=(-\Delta)^{s_1+s_2}$ and, therefore
$$
G_{a,s_1+s_2}=G_{a,s_1}*G_{a,s_2}
$$
using the elementary properties of convolutions. Note  that the result of Corollary \ref{Cor2.Grsm} can be obtained in a straightforward way using the standard PDEs technique, but we prefer to use the explicit formulas \eqref{2.Poisson} and \eqref{2.Poisson0} which look a bit more transparent. In addition, using the Poisson Summation Formula in a more sophisticated way (e.g. in the spirit of \cite{mar}, see also references therein), we can obtain much better (exponentially convergent) series for $M_{0,s}(x)$.
\end{remark}
We are now ready to state and prove the main result of this section.
\begin{theorem}\label{Th2.1c} Let $s>0$. Then
\begin{equation}\label{2.1cesaro}
M_{a,s}=M_{a,s}(\pi,\pi,\pi)=\lim_{N\to\infty}\sum_{n=1}^n\(1-\frac nN\)\frac{(-1)^nr_3(n)}{(a^2+n)^s}
\end{equation}
and, therefore, \eqref{2.sphere} is  first order Cesaro summable by expanding spheres.
\end{theorem}
\begin{proof} As already mentioned above, it is sufficient to consider the case $0<s<1$ only.  We also recall that \eqref{2.F} is nothing more than formal Fourier expansions for the function $M_{a,s}(x)$, therefore, to verify the second equality in \eqref{2.1cesaro}, we need to check the convergence of Fourier expansions of $M_{a,s}(x)$ at $x=(\pi,\pi,\pi)$ by first Cesaro expanding spheres. To do this, we use the analogue of Riemann localization property for multi-dimensional Fourier series. Namely, as proved in \cite{stein}, this localization is satisfied for first order Cesaro summation by expanding spheres in the class of functions $f$ such that
$$
\int_{\Bbb T^3}|f(x)|\ln_+|f(x)|\,dx<\infty
$$
(this is exactly the critical case $\kappa=\frac{d-1}2=1$ for $d=3$).
Thus, since this condition is satisfied for $M_{a,s}(x)$ due to Corollary \ref{Cor2.Grsm}, the Fourier series for $M_{a,s}(x)$ and $M_{a,s}(x)-\frac1{(2\pi)^3}G_{a,s}(x)$ are convergent or divergent simultaneously. Since the second function is $C^\infty$ on the whole torus, we have the desired convergence, see also \cite{MFS} and references therein. Thus, the second equality in \eqref{2.1cesaro} is established.  To verify the first equality, it is enough to mention that the series is second order Cesaro summable to $M_{a,s}$ due to Theorem \ref{Th2.2c}. This finishes the proof of the theorem.
\end{proof}

\section{Concluding remarks}\label{s3}
Note that formally Theorem \ref{Th2.1c} covers Theorem \ref{Th2.2c}. Nevertheless, we would like to present both methods. The one given in subsection \ref{s21} is not only very elementary and transparent, but also can be easily extended to summation by general expanding domains $N\Omega$ where $\Omega$ is a sufficiently regular bounded domain in $\R^3$ containing zero. Also the rate of convergence of second Cesaro sums can be easily controlled. Some numeric simulations for the case of NaCl-Madelung constant ($a=0$, $s=\frac12$) are presented in the figure below
\begin{figure}[h!]
  \centering
    \includegraphics[width=0.8\linewidth]{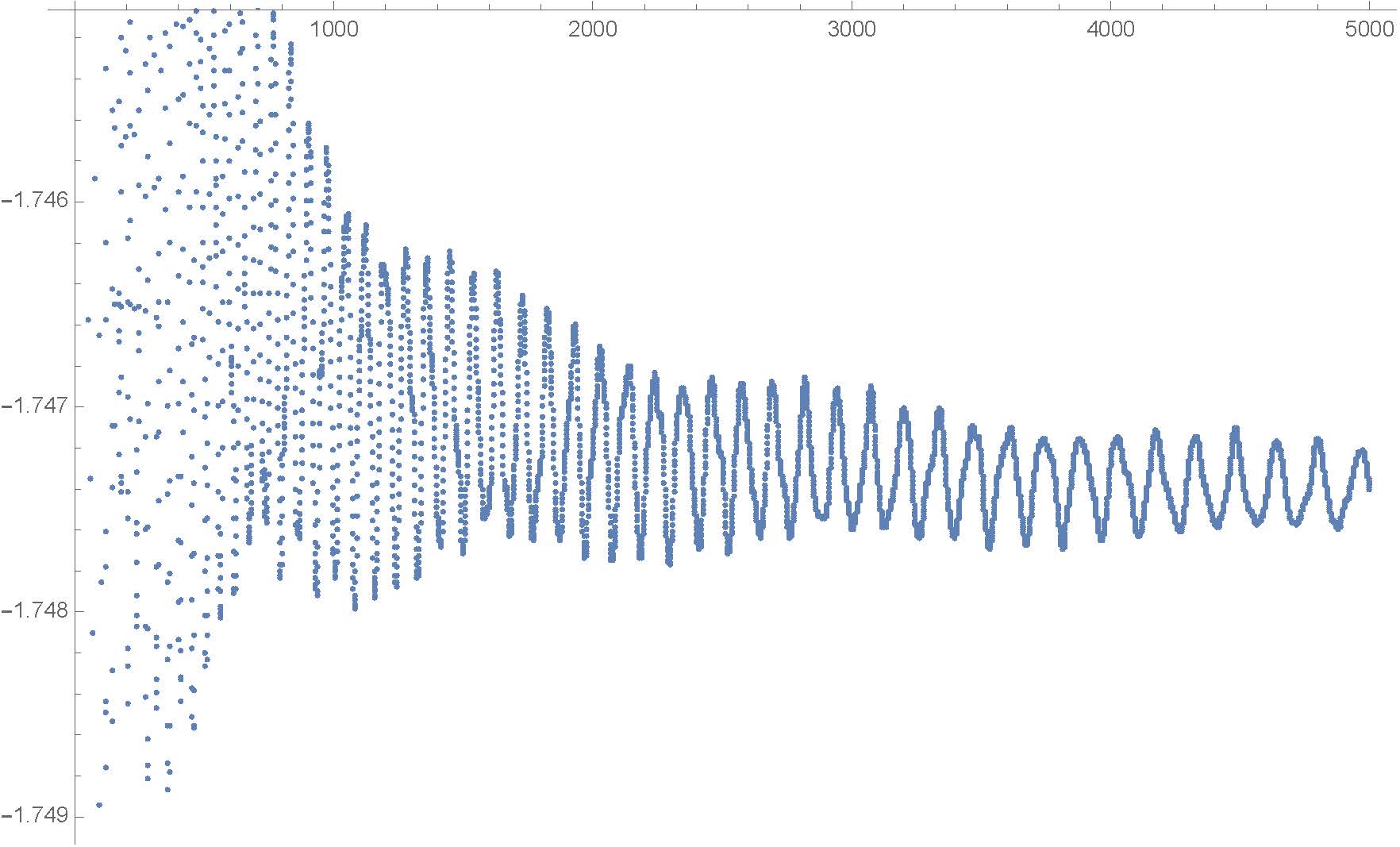}
    \caption{A figure plotting $N$th partial sums of \eqref{2.2good} with $a=0$ and $s=\frac12$ up to N = 5000.}
  \label{fig:coffee}
\end{figure}

\noindent and we clearly see the convergence to the Madelung constant
$$
M_{0,1/2}=-1.74756...
$$
The second method (used in the proof of Theorem \ref{Th2.1c}) is more delicate and strongly based on the Riemann localization for multiple Fourier series and classical results of \cite{stein}.  This method is more restricted to expanding spheres and the rate of convergence is not clear. Some numeric simulation for the NaCl-Madelung constant is presented in the figure below
\begin{figure}[h!]
  \centering
    \includegraphics[width=\linewidth]{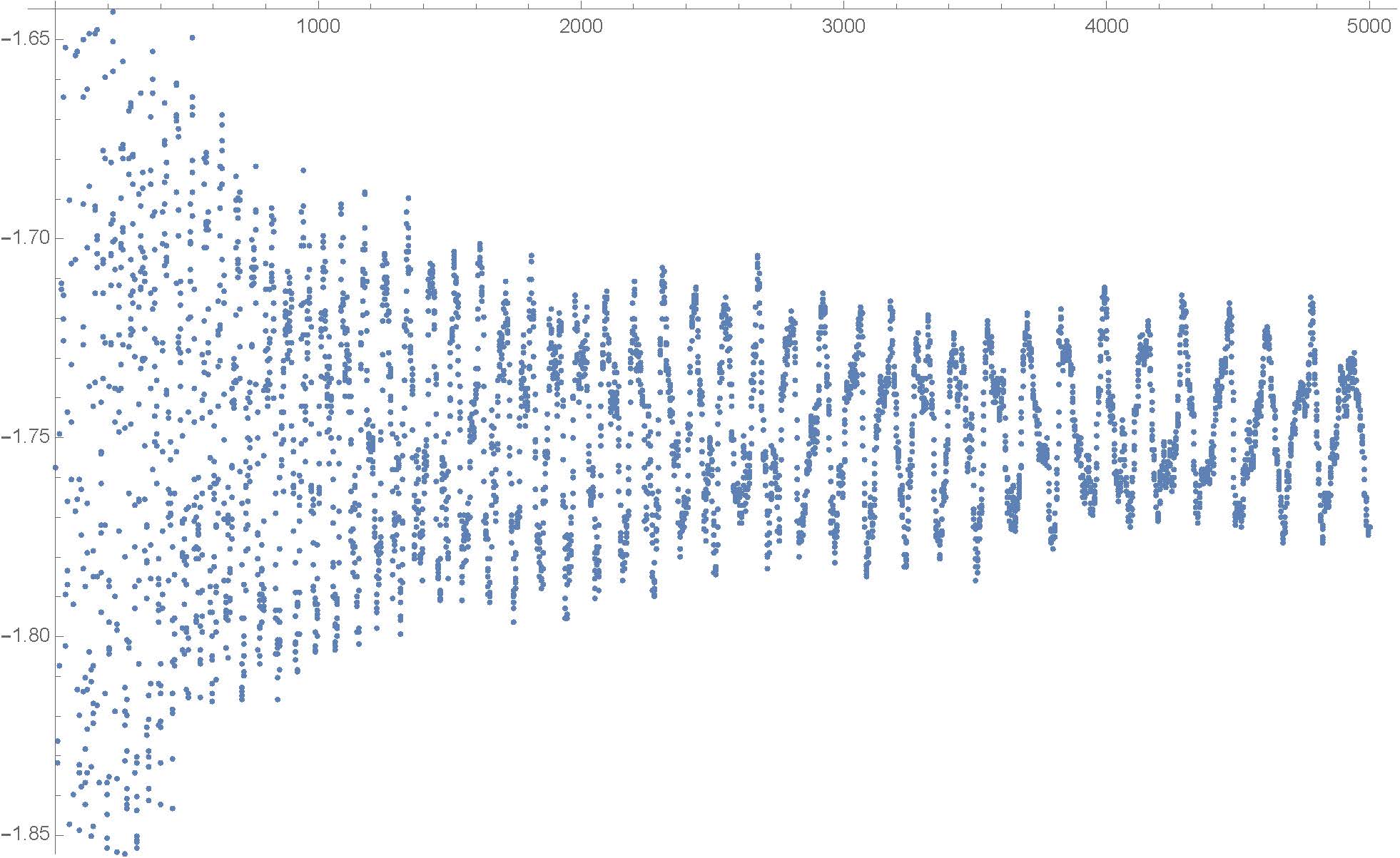}
    \caption{A figure plotting $N$th partial sums of \eqref{2.1cesaro} with $a=0$ and $s=\frac12$ up to N = 5000.}
  \label{fig:coffee1}
\end{figure}

\noindent and we see that the rate of convergence is essentially worse than for the case of second order Cesaro summation. As an advantage of this method, we mention the ability to extend it for more general class of exponential sums of the form \eqref{2.F}.
\par
Both methods are easily extendable to other  dimensions $d\ne3$. Indeed, it is not difficult to see that the elementary method works for Cesaro summation of order $\kappa\ge d-2$ and the second one requires weaker assumption $\kappa\ge\frac{d-1}2$. Using the fact that the function $M_{a,s}(x)$ is more regular (belongs to some Sobolev space $W^{\eb,p}(\Bbb T^3)$), together with the fact that Riemann localization holds for slightly subcritical values of $\kappa$ if this extra regularity is known (see e.g. \cite{MFS}), one can prove  convergence for some $\kappa=\kappa(s)<\frac{d-1}2$ although the sharp values for $\kappa(s)$ seem to be unknown.

\end{document}